\documentclass[11pt,letterpaper]{amsart}
\usepackage{fancyhdr}
\usepackage{bm}
\usepackage{hyperref}
\usepackage{graphicx} 
\usepackage{nicematrix}
\usepackage{amsthm,amssymb,amsmath}
\usepackage[T1]{fontenc}
\usepackage[utf8]{inputenc}
\usepackage{hyperref}
\usepackage{lipsum}
\usepackage{mathrsfs}
\usepackage{enumitem}
\usepackage{stmaryrd}
\usepackage{setspace}
\usepackage{yfonts}
\usepackage{float}
\usepackage{mathtools}
\usepackage{parskip}
\usepackage[all,cmtip]{xy}
\usepackage{tikz-cd}
\tikzcdset{row sep/normal=50pt, column sep/normal=50pt}

\newtheorem{lemma}{Lemma}[section]
\newtheorem{remark}[lemma]{Remark}
\newtheorem{theorem}[lemma]{Theorem}
\newtheorem{theorem*}{Theorem}
\newtheorem{example}[lemma]{Example}
\newtheorem{example*}[lemma]{Example}

\newtheorem{manualtheoreminner}{Theorem}
\newenvironment{manualtheorem}[1]{%
  \renewcommand\themanualtheoreminner{#1}%
  \manualtheoreminner
}{\endmanualtheoreminner}
\usepackage{blindtext}
\usepackage[a4paper, total={6in, 8in}]{geometry}
\usepackage[backend=biber, style=numeric, sorting=nyt]{biblatex}
\title{Smooth blow up structures on projective bundles over projective spaces}
\author{\small{Supravat Sarkar}}
\date{}
\begin{document}
\maketitle
\begin{abstract}
    Assuming Hartshorne's conjecture on complete intersections, we classify projective bundles over projective spaces which has a smooth blow up structure over another projective space. Under some assumptions, we also classify projective bundles over projective spaces which has a smooth blow up structure over some arbitrary smooth projective variety, not necessarily a projective space. We verify which of the globally generated vector bundles over projective space of first Chern class at most five has the property that their projectivisation has a smooth blow up structure, with no additional assumption. In the way, we get some new examples of varieties with both projective bundle and smooth blow up structures.
\end{abstract}
\begin{center}
\end{center}
\begin{center}
\end{center}

\section{Introduction}
 Given a smooth projective variety $X$, there are two standard ways of constructing another smooth projective variety with Picard number one more than $X$. One is to construct a projective bundle over $X$, another is to blow-up a smooth subvariety of codimension at least $2$ in $X$. It is interesting to consider when a smooth projective variety can be constructed by the above procedure in two different ways, in other words, when a smooth projective variety has two different structures of the above kind. We are particularly interested in the case of smooth projective varieties of Picard rank $2$ having two different such structures.
 
  There are a few classes of examples of such varieties in literature. In \cite{watanabe2014p1},\cite{occhetta2022manifolds} there are examples of smooth projective varieties having two projective bundle structures. In \cite{sarkar2024varieties},\cite{ein1989some}, \cite{crauder1989cremona}, \cite{crauder1991cremona} there are examples of smooth projective varieties having two smooth blow up structures. In \cite{bansal2024extremal},\cite{vats2023correspondence},\cite{galkin2022projective}, \cite{ray2020projective}, there are examples of smooth projective varieties having both projective bundle and smooth blow up structures. In \S 3 of this paper, we give a few examples of projective bundles over projective spaces which has a smooth blow up structure. Some of them have already appeared in literature, some of them are new.

  Next, we turn our attention to the classification of such varieties. \cite{sato1985varieties} classifies varieties with two projective bundle structures over projective spaces. \cite{sarkar2024varieties} gives a satisfactory answer towards the classification of varieties with two smooth blow up structures over projective spaces. In this paper, we tackle the remaining case, of varieties with both projective bundle and smooth blow up structures over projective spaces.

  \begin{manualtheorem}{A} \label{theorem A} Let $n,r$ be positive integers, $X$ a smooth projective variety with two contractions $\pi: X\to \mathbb{P}^n$, $\phi: X\to \mathbb{P}^{n+r}$ such that $\pi$ is a projective bundle structure and $\phi$ is a smooth blow up along a nonlinear subvariety. Let $d\geq 2$ be the integer such that the rational map $\pi \circ \phi^{-1}:\mathbb{P}^{n+r}\dashrightarrow \mathbb{P}^{n}$ is defined by an $(n+1)$-tuple of homogeneous polynomials of degree $d$.
\begin{enumerate}
    \item If $d=2$, then $X$ is as in example \ref{1} or \ref{2} of \S 3.
    \item If $d\geq 3$ then either $d=3$ and $X$ is as in example \ref{3} of \S 3, or $r=n/d$ and $W$ has same Betti numbers as $\mathbb{P}^{n-1}.$
    \item If Hartshorne's conjecture on complete intersection holds then the 2nd case in $(2)$ is not possible, so $X$ is always as in one of the examples of \S 3.
\end{enumerate}
\end{manualtheorem}
By Hartshorne's conjecture on complete intersection, we mean the following statement: If $W$ a smooth subvariety of $\mathbb{P}^n$ of dimension $>2n/3$, then $W$ is a complete intersection.

Next, we want to classify all projective bundles over projective spaces which has a smooth blow up structure. This smooth blow up structure can be  over any smooth projective variety, not necessarily a projective space. To the author's best knowledge, there is no known example other than the examples in \S 3. We prove this in Theorem \ref{theorem:B} under several assumptions.
As in \cite{anghel2013globally}, for a globally generated vector bundle $F$ on $\mathbb{P}^n$ let $P_F$ be the dual of the kernel of the evaluation morphism $H^0(F)\otimes \mathcal{O}_{\mathbb{P}}\to F.$
Let $M_{n,c}$ be the following statement

$M_{n,c}:$ If $E$ is a globally generated vector bundle on $\mathbb{P}^n$ with first Chern class $\leq$ min$\{c,n-1\}$ and $H^0(E^*)=H^1(E^*)=0$, then one of the following holds:
\begin{enumerate}
    \item $E\cong A\oplus P_B$, where $A, B$ are globally generated split bundles.
    \item $n\geq 3$ and $E\cong\Omega_{\mathbb{P}^n}(2).$
         \item $n\geq 4$ and $E=\wedge^2(T_{\mathbb{P}^n}(-1))$.
\end{enumerate}
By \textit{split} bundle, we mean a direct sum of line bundles.

In {\cite[Conjecture 0.3]{anghel2013globally}}, it is conjectured that $M_{n,c}$ is true for all $n$ and $c$. By \cite{anghel2013globally},\cite{anghel2018globally}, \cite{anghel2020globally}, $M_{n,5}$ is true for all $n$.

Now we can state our result.
\begin{manualtheorem}{B} \label{theorem:B}

    Let $n,r$ be positive integers, $X$ a smooth projective variety which is a $\mathbb{P}^r$-bundle over $\mathbb{P}^n$, and also has a smooth blow up structure over a smooth projective variety $Y$. Let $\mathcal{E}$ be the unique nef vector bundle of rank $r+1$  over $\mathbb{P}^n$ such that $X\cong \mathbb{P}_{\mathbb{P}^n}(\mathcal{E})$ over $\mathbb{P}^n$ and $\mathcal{E}(-1)$ is not nef. Suppose $\mathcal{E}$ is not ample. Then $c_1(\mathcal{E})\leq n$ and the following holds:
    \begin{enumerate}
        \item Suppose $\mathcal{E}$ is globally generated, equivalently, the ample generator of $Y$ is globally generated. If $M_{n,c_1(\mathcal{E})}$ is true, then either $X$ is one of the examples in \S 3 or $c_1=n$, codim $W=2$, and if $0\to \mathcal{O}_{\mathbb{P}^n}^k\to F\to \mathcal{E}\to 0 $ a short exact sequence of vector bundles on $\mathbb{P}^n$ with $H^0(F^*)=H^1(F^*)=0$, then $F$ is indecomposable.
        \item Suppose the ample generator of $Y$ is very ample. Then either $X$ is one of the examples in \S 3 or rk $\mathcal{E}\leq n$.
    \end{enumerate}
\end{manualtheorem}

The nonampleness of $\mathcal{E}$ is an assumption made in the previous works also, see for example {\cite[Lemma 2.5]{li2024projective}}, also the equivalent conditions in {\cite[Proposition 5]{li2021projective}} in our setup is equivalent to nonampleness of $\mathcal{E}$.

Finally, we classify the projectivizations of globally generated vector bundles over projective spaces of first Chern class at most $5$ which has a smooth blow up structure, with no additional assumptions. This is an extension of {\cite[Theorem E]{bansal2024extremal}}.
\begin{manualtheorem}{C} \label{theorem:C}
    Let $n$ be a positive integer and $\mathcal{E}$ a globally generated vector bundle on $\mathbb{P}^n$ with $c_1(\mathcal{E})\leq 5$. If $X=\mathbb{P}_{\mathbb{P}^n}(\mathcal{E})$ is a smooth blow up, then $X$ is one of the examples in \S 3.
\end{manualtheorem}
As mentioned in \cite{bansal2024extremal}, this result is a counterpart of the results in \cite{li2021projective} and \cite{li2024projective} in the sense that \cite{li2021projective} and \cite{li2024projective} classify when the blow-up of a variety of small dimensions in projective space has a projective bundle structure, and our result classifies when the projectivization of a globally generated vector bundle of small first Chern class over projective space has a smooth blow up structure.

 The same method in this paper also classifies projective bundles over projective spaces that have another projective bundle structure, under some assumptions and assuming $M_{n,c_1}$ (see Remark \ref{rmk}). For some previously obtained results regarding these varieties, see \cite{lanteri1988projective}.

\section{Notations, conventions and lemmas}
We work throughout over the field $k=\mathbb{C}$ of complex numbers. A \textit{variety} is an integral, separated scheme of finite type over $k$. By subvariety we always mean a closed subvariety. By  \textit{smooth blow up} we mean blow up of a smooth variety along a smooth subvariety. For positive integers $m,n$, $Gr(m,n)$ denotes the Grassmanian variety of $m$-dimensional linear subspaces of $k^n.$ If $\mathcal{E}$ is a vector bundle of rank $\geq 2$ on $\mathbb{P}^n$ and the $R$ is the extremal ray of $\mathbb{P}(\mathcal{E})$ which is not contracted by the projective bundle map, then the contraction of $R$, if it exists, will be called the \textit{other contration} of $\mathbb{P}(\mathcal{E})$.

We shall need the following lemmas.
\begin{lemma}\label{lemma:fiber}
     Let $n,r$ be positive integers and for $1\leq i\leq r$, $E_i$ a globally generated vector bundle over $\mathbb{P}^n$. Let $V_i=H^0(\mathbb{P}^n, E_i)$. Let $\phi:\mathbb{P}_{\mathbb{P}^n}(\oplus_iE_i)\to\mathbb{P}(\oplus_iV_i)$ be the morphism given by $|\mathcal{O}_{\mathbb{P}(\oplus_iE_i)}(1)|$, $\pi:\mathbb{P}_{\mathbb{P}^n}(\oplus_iE_i)\to\mathbb{P}^n$ the projection. For $1\leq i\leq r$, let $\phi_i:\mathbb{P}_{\mathbb{P}^n}(E_i)\to\mathbb{P}(V_i)$ be the morphism given by $|\mathcal{O}_{\mathbb{P}(E_i)}(1)|$, $\pi_i:\mathbb{P}_{\mathbb{P}^n}(E_i)\to\mathbb{P}^n$ the projection. Let $v_i\in V_i^*$ for $1\leq i\leq r$, $v=(v_i)_i\in (\oplus_iV_i)^*$. Then we have $\phi^{-1}([v])\cong \pi\phi^{-1}([v])=\cap_i S_i$, where \begin{equation*}
         S_i=\begin{cases}
         \pi_i\phi_i^{-1}([v_i]) & \quad \text{if } v_i\neq 0\\
         \mathbb{P}^n & \quad \text{if } v_i= 0.
     \end{cases}
     \end{equation*}
 \end{lemma}
 \begin{proof}
     Same argument as in the proof of {\cite[Lemma 2.4]{bansal2024extremal}} shows the result.
 \end{proof}

\begin{lemma}\label{lemma:kollar}
     Let $Y$ be a smooth projective variety, $W$ a smooth subvariety of $Y$, $X\xrightarrow{\phi}Y$ the blow of $Y$ along $W$, $E\hookrightarrow X$ the exceptional divisor. Let $W'$ be a closed subscheme of $Y$ such that $\phi^{-1}(W')=E$ scheme-theoretically. Then $W=W'$.
     \end{lemma}
     \begin{proof}
         Clearly $W'_{red}=W$. Let $\mathcal{I}', \mathcal{I}, \mathcal{I}_E$ be the ideal sheaves of $W', W, E$ respectively. Let $\mathcal{I}''=\mathcal{I}^2+\mathcal{I}'$, $W''=$ the closed subscheme of $Y$ corresponding to $\mathcal{I}''$. Since $\mathcal{I}'\subset \mathcal{I}''\subset \mathcal{I},$ we have $W'\supset W''\supset W$, so $\phi^{-1}W''=E$. If $W''=W$, then $\mathcal{I}''=\mathcal{I}$ so by Nakayama $\mathcal{I}=\mathcal{I}'$, that is, $W=W'$. So we can assume without loss of generality $\mathcal{I}'\supset \mathcal{I}^2$.

         So $\mathcal{F}=\mathcal{I}'/\mathcal{I}^2, \mathcal{E}=\mathcal{I}/\mathcal{I}^2, F=\mathcal{I}/\mathcal{I}'$ are coherent sheaves on $W$ with $\mathcal{E}$ locally free, and they fit into a short exact sequence $0\to \mathcal{F}\to \mathcal{E}\to F\to 0$. We want to show $F=0$. 

         Since $\phi^{-1}W'=E$, the map $\phi^*\mathcal{I}'\to \phi^*\mathcal{O}_X=\mathcal{O}_Y$ has image $\mathcal{I}_E$. Clearly, for each $k\geq 1$, the map $\phi^*\mathcal{I}^k\to \mathcal{O}_Y$ has image $\mathcal{I}_E^k$. So we have a commutative square with vertical arrows surjective:
         \begin{center}
\begin{tikzcd}
 \phi^*\mathcal{I}^2 \arrow[r, "\alpha"   ] \arrow[d]
&  \phi^*\mathcal{I}'\arrow[d] \\
\mathcal{I}_E^2\arrow[r, "\beta"]
& |[, rotate=0]|  \mathcal{I}_E.
\end{tikzcd}
\end{center}
 This induces a surjection coker $\alpha\to$ coker $\beta.$ Clearly coker $\beta=\mathcal{I}_E/\mathcal{I}_E^2=\mathcal{O}_{\mathbb{P}(\mathcal{E})}(1)$. Applying $\phi^*$ to the short exact sequence $0\to \mathcal{I}^2\to \mathcal{I}'\to \mathcal{F}\to 0$ we get coker $\alpha=\phi^*\mathcal{F}$. So we have surjection $\phi^*\mathcal{F}\xrightarrow{f}\mathcal{O}_{\mathbb{P}(\mathcal{E})}(1).$ Restricting to $\phi^{-1}(w)$ for $w\in W$ we get surjection $\mathcal{O}_{\mathbb{P}(\mathcal{E}(w))}\otimes_k \mathcal{F}(w)\to \mathcal{O}_{\mathbb{P}(\mathcal{E}(w))}(1).$ Since for any finite dimensional vector space $V$ over $k$, no proper subspace of $H^0(\mathbb{P}(V),\mathcal{O}_{\mathbb{P}(V)}(1))$ generate $\mathcal{O}_{\mathbb{P}(V)}(1)$, we see that the induced map $\mathcal{F}(w)=H^0(\mathbb{P}(\mathcal{E}(w)), \mathcal{O}_{\mathbb{P}(\mathcal{E}(w))}\otimes_k \mathcal{F}(w))\to H^0(\mathbb{P}(\mathcal{E}(w)),\mathcal{O}_{\mathbb{P}(\mathcal{E}(w))}(1))=\mathcal{E}(w)$ is surjective. But the short exact sequence $0\to \mathcal{F}\to \mathcal{E}\to F\to 0$ gives exact sequence $\mathcal{F}(w)\to \mathcal{E}(w)\to F(w)\to 0$. So, $F(w)=0$ for all $w\in W$. By Nakayama, $F=0$. So, $\mathcal{I}'=\mathcal{I},$ that is $W=W'.$
     \end{proof}

 Now let us consider the following setup.
 
\textbf{Setup:} Let $n, r$ be positive integers, $Y$ a smooth projective variety of dimension $n+r$, $W$ a smooth subvariety of $Y$ of dimension $m\leq n+r-2$. Let $\phi:X \to Y$ be the blow up of $Y$ along $W$, $E\subset X$ the exceptional divisor of $\phi$. Suppose $X$ has also a structure of a $\mathbb{P}^r$-bundle over $\mathbb{P}^n$, let $\pi:X\to \mathbb{P}^n$ be a $\mathbb{P}^r$-bundle structure. There is a unique nef vector bundle $\mathcal{E}$ of rank $r+1$  over $\mathbb{P}^n$ such that $X\cong \mathbb{P}_{\mathbb{P}^n}(\mathcal{E})$ over $\mathbb{P}^n$ and $\mathcal{E}(-1)$ is not nef. (Recall that a vector bundle $G$ over $\mathbb{P}^n$ is called nef/ample/big if $\mathcal{O}_{\mathbb{P}(G)}(1)$ is nef/ample/big). Since $\mathbb{Z}^2\cong\text{Pic }X\cong \mathbb{Z}\oplus \text{Pic }Y$, we see that $\text{Pic }Y\cong \mathbb{Z}$. Since $Y$ is also rational, $Y$ is a Fano manifold of Picard number $1$. Let $\mathcal{O}_Y(1)$ be the ample generator of $\text{Pic }Y$, and $H_1=\pi ^*\mathcal{O}_{\mathbb{P}^n}(1)$, $H_2=\phi^*\mathcal{O}_Y(1)$, $U=\mathcal{O}_{\mathbb{P}(\mathcal{E})}(1)\in \text{Pic }(X)$. We can regard the Weil divisor $E$ also as an element of Pic$(X)$.  Let $\tau$ be the index of $Y$, that is, $-K_Y=\mathcal{O}_Y(\tau). $ Also, let $\alpha$ be the top self-intersection number of $\mathcal{O}_Y(1)$, so $H_2^{n+r}=\alpha.$ Let $\mathcal{O}_W(1)=\mathcal{O}_Y(1)|_{W}.$

Let $F_1$ be a line in a fiber $\pi^{-1}(z)$ (we have $\pi^{-1}(z)\cong \mathbb{P}^r)$, and $F_2$ be a line in a fiber $\phi|_E^{-1}(w)$ (we have $\phi|_E^{-1}(w)=\mathbb{P}^{n+r-m-1}$).  So, $N_1(X)_{\mathbb{Q}}$ has basis $\{F_1, F_2\}$, and $\{H_1, U\}$ and $\{H_2, E\}$ are both bases of Pic$(X)\cong \mathbb{Z}^2$. We have $H_1. F_1=H_2.F_2=0$, $U.F_1=1$, $E.F_2=-1$. Let $a=H_1.F_2$, $d=E.F_1$, $b=U.F_2$. By {\cite[Lemma 3]{li2021projective} and its proof, we have $H_2. F_1=a$, $a|1+bd$, and 
 $$H_1=dH_2-aE$$

        $$U=\frac{1+bd}{a}H_2-bE$$

        in Pic($X$).
Note that if $Y=\mathbb{P}^{n+r}$, then $H_1=dH_2-aE$ implies that the rational map $\pi \circ \phi^{-1}:\mathbb{P}^{n+r}\dashrightarrow \mathbb{P}^{n}$ is defined by an $(n+1)$-tuple of homogeneous polynomials of degree $d$, so this $d$ is the same $d$ as in the statement of Theorem \ref{theorem A}. Since $U$ is nef but $U-H_1$ is not nef, we have $0\leq b<a$.

 For $0\leq i\leq n$ Let $c_i, s_i\in \mathbb{Z}$ denote the $i$-th Chern and Segre classes of $\mathcal{E}$, respectively. Here we are using the natural identification $H^{2i}(\mathbb{P}^n, \mathbb{Z})=\mathbb{Z}$, via the natural generator of $H^{2i}(\mathbb{P}^n, \mathbb{Z})$ given by $i$-th power of the first Chern class of $\mathcal{O}_{\mathbb{P}^n}(1).$ By definition of Segre classes as in {\cite[Chapter 3]{fulton2013intersection} we have $U^{r+i}H_1^{n-i}=(-1)^i s_i$ for $0\leq i\leq n.$ Also by{\cite[Chapter 3]{fulton2013intersection} we have 
 \begin{equation}
     \Big(\sum_{i=0}^n c_it^i\Big)\Big(\sum_{i=0}^n s_it^i\Big)\equiv 0(\text{mod } t^{n+1}).
 \end{equation}
 \begin{lemma}\label{lemma:a=1}
 In the notations of the setup, we have:
 \begin{enumerate}
     \item[(i)] We have: $\mathcal{E}$ is not ample $\Leftrightarrow a=1, b=0.$
     \item[(ii)] We have:$$\tau=d(n+1-c_1)+(r+1)\frac{1+bd}{a},$$
     $$n+r-m-1=a(n+1-c_1)+b(r+1).$$
     \item[(iii)] If $\mathcal{E}$ is not ample, we have $d=\frac{\tau-r-1}{n+r-m-1}, c_1=m-r+2.$
     \item[(iv)] If $\mathcal{E}$ is not ample, then $W$ is the base locus of $\pi \circ \phi^{-1}:Y\dashrightarrow \mathbb{P}^{n}$.
     \item[(v)] $a^r|\alpha.$
     \item[(vi)] If $Y=\mathbb{P}^{n+r}$, then $a=1, b=0, \text{codim }W=n+r-m=\frac{n}{d}+1, c_1=\frac{d-1}{d}n+1.$
 \end{enumerate}
 \end{lemma}
 \begin{proof}
    \textit{(i):} Note that: $\mathcal{E}$ is not ample $\Leftrightarrow \mathcal{O}_{\mathbb{P}(\mathcal{E})}(1)(=U)$ is not ample $\Leftrightarrow U.F_2\leq 0$ (as $U.F_1=1>0$) $\Leftrightarrow b\leq 0 \Leftrightarrow b=0$ (as we assumed $b\geq 0$) $ \Leftrightarrow a=1, b=0$ (as $a|1+bd$).  
    
    \textit{(ii):} By canonical bundle formulae, $\tau H_2-(n+r-m-1)E=-K_X=(n+1-c_1)H_1+(r+1)U=(n+1-c_1)(dH_2-aE)+(r+1)(\frac{1+bd}{a}H_2-bE).$ Comparing coefficients of $H_2$ and $E$ in both sides, we get $ii)$.

\textit{(iv):} We have $a=1, b=0$ by $(i)$. Let $W'$ be the base locus of $\pi \circ \phi^{-1}$. Since $dH_2=H_1+E$, we see that $\phi^{-1}(W')=E$ scheme-theoretically. By Lemma \ref{lemma:kollar}, $W=W'$.

\textit{(v):} We have $1=H_2^{n+r}=(-bH_1+aU)^{n+r}=\sum_{i=0}^{n} {{n+r}\choose i}(-b)^ia^{n+r-i}H_1^iU^{n+r-i}.$ Clearly each term in this sum is divisible by $a^r$.

$(iii)$ is immediate from $(i)$ and $(ii)$. $(vi)$ is immediate from $(v),(i)$ and $(iii).$ 
    
 \end{proof}
 \begin{lemma}\label{lemma:adivide}
  In the notations of the setup, assume $n\geq 2.$
 \begin{enumerate}
     \item[(i)] If $1\leq t \leq min\{r,n+r-m-1\},$ $j\geq 0$ are integers, such that $a^j|{{n+r-t}\choose{n-1}}$, then $a^{min\{2,j+1\}}|{{n+r-t}\choose{n}}.$
     \item[(ii)] $a|{{n+r-1}\choose{n}}.$
     \item[(iii)] If $n+r-m\geq 3$, and $r\geq 2$, then $a^2|{{n+r-2}\choose{n}}.$
     \item[(iii)] $a^3|{{n+r-1}\choose{n}}b^2d+{{n+r-1}\choose{n-1}}ab(\frac{1+bd}{a}-c_1d)+{{n+r-1}\choose{n-2}}a^2(dU^{r+2}H_1^{n-2}-c_1.\frac{1+bd}{a}).$
 \end{enumerate}
     
 \end{lemma}
 \begin{proof}

     \textit{(i):} Since $n+r-t>m=\text{dim }W$,we have $$0=E^tH_2^{n+r-t}=(-\frac{1+bd}{a}H_1+dU)^t(-bH_1+aU)^{n+r-t}.$$ So,
     \begin{equation}
\begin{split}
0 = & \binom{n+r-t}{n} (-b)^{n}a^{r-t}d^{t} \\
& +\Big(\sum_{i=0}^{t}\binom{t}{i}\Big(-\frac{1+bd}{a}\Big)^{i}d^{t-i}H_{1}^{i}U^{t-i}\Big) \cdot \Big(\sum_{i=0}^{n-1}\binom{n+r-t}{i}(-b)^{i}a^{n+r-t-i} H_{1}^{i}U^{n+r-t-i}\Big).
\end{split}
\end{equation}

         Note that $a^{r-t+min\{2,j+1\}}|\binom{n+r-t}{i}(-b)^{i}a^{n+r-t-i} H_{1}^{i}U^{n+r-t-i}$ for all $0\leq i\leq n-1$. So,  right hand side of equation $(2)$ is $\equiv {{n+r-t}\choose{n}}(-b)^na^{r-t}d^t(\text{mod }a^{r-t+min\{2,j+1\}}).$ So, $$a^{r-t+min\{2,j+1\}})|{{n+r-t}\choose{n}}(-b)^na^{r-t}d^t,$$ hence $a^{min\{2,j+1\}})|{{n+r-t}\choose{n}}(-b)^nd^t.$ As $a|1+bd$, so gcd$(a,b)=$gcd$(a,d)=1$. So $(i)$ follows.

         $(ii)$ follows by putting $t=1, j=0$ in $(i)$.

         \textit{(iii):} Put $j=0$ in $(i)$ to get $a|{{n+r-t}\choose{n}}$ for $t=1,2$. So, $a|{{n+r-1}\choose{n}}-{{n+r-2}\choose{n}}={{n+r-2}\choose{n-1}}.$ Now put $t=2, j=1$ in $(i)$ to get $(iii)$.

         \textit{(iv):} Put $t=1$ in equation $(2)$ and get $$0={{n+r-1}\choose{n}}(-b)^na^{r-1}d+\sum_{i=0}^{n-1} {{n+r-1}\choose{i}}(-b)^i a^{n+r-1-i}(dU^{n+r-i}H_1^{i}-\frac{1+bd}{a}U^{n+r-i-1}H_1^{i+1})$$ 
         $$\hspace{-100 pt}\equiv{{n+r-1}\choose{n}}(-b)^na^{r-1}d+{{n+r-1}\choose{n-1}}(-b)^{n-1}a^{r}(c_1d-\frac{1+bd}{a})$$
         $$+{{n+r-1}\choose{n-2}}(-b)^{n-2}a^{r+1}(dU^{r+2}H_1^{n-2}-\frac{1+bd}{a}c_1) (\text{mod }a^{r+2}).$$ Here we used $U^rH_1^n=1, U^{r+1}H_1^{n-1}=-s_1=c_1.$ Since gcd$(a,b)=1$, we can take the common factor $(-b)^{n-2}$ out, and then divide by $a^{r-1}$ to get $(iv)$.
         
 \end{proof}
 
 \begin{lemma} \label{lemma:segre}
In the notations of the setup, suppose $\mathcal{E}$ is not ample. Then for $c_1-1\leq i\leq n,$ we have $(-1)^{i}s_{i}=d^{n-i}\alpha$ and $(-1)^{c_1-2}s_{c_1-2}=d^{n-c_1+2}\alpha-\mathcal{O}_W(1)^{.m}<d^{n-c_1+2}\alpha.$
 \end{lemma}
 \begin{proof}
 We have $a=1, b=0$ by Lemma \ref{lemma:a=1}. Since $E\to W$ is a $\mathbb{P}^{n+r-m-1}$-bundle, we have $E^{n+r-l}H_2^l=0$ for $m<l<n+r$ and $E^{n+r-m}H_2^m=(-1)^{n+r-m-1}\mathcal{O}_W(1)^{.m}$ (see also {\cite[Lemma 2.4]{sarkar2024varieties}}. So, $(-1)^{i} s_{i}=U^{r+i}H_1^{n-i}=(dH_2-E)^{n-i}H_2^{r+i}=d^{n-i}\alpha$, as $H_2^{n+r}=\alpha$ and for $i\geq c_1-1$ we have $r+i>c_1+r-2=m=\text{dim }W$, so $E^{n-i-j}H_2^{r+i+j}=0$ for $0\leq j <n-i$. Also, $(-1)^{c_1-2}s_{c_1-2}=U^{r+c_1-2}H_1^{n-c_1+2}=(dH_2-E)^{n-c_1+2}H_2^{r+c_1-2}=(dH_2-E)^{n+r-m}H_2^{m}=d^{n+r-m}\alpha+(-1)^{n+r-m}E^{n+r-m}H_2^m=d^{n-c_1+2}\alpha-\mathcal{O}_W(1)^{.m}.$ The last inequality in Lemma follows from the fact that $\mathcal{O}_W(1)$ is ample.
 \end{proof}
 
 For a compact complex manifold $Z$ and an integer $i$, denote dim$_{\mathbb{C}}$$H^{i}(Z, \mathbb{C})$ by $h^i(Z)$.
 Let $a_i=h^{2i}(W)$. Let $P(t)=\sum_i a_i t^i\in \mathbb{Z}[t].$
 \begin{lemma}\label{lemma:P}
     In the notations of the setup, suppose $Y=\mathbb{P}^{n+r}$. We have: \begin{enumerate}
         \item[(i)] $a_i\neq 0 \Longleftrightarrow 0\leq i\leq m$
         \item[(ii)] $P(t)=(\sum_{i=0}^{d-1} t^{in/d})(\sum_{i=0}^{r-1} t^i).$
         \item[(iii)] For each $k$, $H^{2k}(W,\mathbb{Z})$ is free abelian.
         \item[(iv)] $n/d\leq r.$
     \end{enumerate}
 \end{lemma}
 \begin{proof}
     $(i)$ follows from the general fact that for a compact Kahler manifold $Y$ of complex dimension $m$, we have $H^{2i}(Y,\mathbb{C})\neq 0$ for $0\leq i\leq m$.
 
 $\underline{(ii)}$ and $\underline{(iii)}$: By {\cite[Theorem 7.31]{MR2451566}}, and Lemma \ref{lemma:a=1}  (ii),  $$H^{2k}(X,\mathbb{Z})\cong H^{2k}(\mathbb{P}^{n+r},\mathbb{Z})\oplus(\oplus_{i=0}^{n/d}H^{2k-2i-2}(W,\mathbb{Z})).$$ Also, we have $H^{2k}(X,\mathbb{Z})=H^{2k}(\mathbb{P}^n\times \mathbb{P}^r, \mathbb{Z})$, as $X$ is a $\mathbb{P}^r$-bundle over $\mathbb{P}^n.$ So, $H^{2k}(W,\mathbb{Z})$ is free abelian, being a subgroup of the free abelian group $H^{2k}(X,\mathbb{Z})$.
 Also, for any $k$ we have $$h^{2k}(\mathbb{P}^n\times \mathbb{P}^r)=h^{2k}(\mathbb{P}^{n+r})+\sum_{i=1}^{n/d} a_{k-i}.$$ Hence, $$\sum_k h^{2k}(\mathbb{P}^n\times \mathbb{P}^r)t^k=\sum_k h^{2k}(\mathbb{P}^{n+r})t^k +\sum_{i=1}^{n/d}(\sum_k a_{k-i}t^{k-i})t^i,$$
 i.e, $$(\sum_{i=0}^n t^i)(\sum_{i=0}^n t^i)=(\sum_{i=0}^{n+r}t^i)+\sum_{i=1}^{n/d}P(t)t^i=(\sum_{i=0}^{n+r}t^i)+tP(t).\sum_{i=0}^{n/d-1}t^i,$$
 i.e, $$\frac{(t^{n+1}-1)(t^{r+1}-1)}{(t-1)^2}=\frac{(t^{n+r+1}-1)}{(t-1)} +tP(t)\frac{(t^{n/d}-1)}{(t-1)},$$
 i.e, $$(t^{n+1}-1)(t^{r+1}-1)=(t-1)(t^{n+r+1}-1)+t(t-1)(t^{n/d}-1)P(t),$$
 i.e., $$P(t)=\frac{(t^n-1)(t^r-1)}{(t^{n/d}-1)(t-1)}.$$ This proves \textit{(ii)}.
 
 \underline{(iv):} Suppose $n/d>r.$ As $d\geq 2,$ Lemma \ref{lemma:P} (ii)
 shows $a_{r-1}>0,$ $a_{n/d}>0$, $a_r=0$. This is impossible by  Lemma \ref{lemma:P} (i).
 \end{proof}
 \begin{lemma} \label{lemma:example}
 Let $n$ be a positive integer, $F$ a globally generated vector bundle over $\mathbb{P}^n$ with first Chern class $c_1\leq n, \mathbb{P}(F)\xrightarrow{\pi} \mathbb{P}^n$ the projection. Let $\phi:\mathbb{P}(F)\to \mathbb{P}(H^0(F))$ the morphism given by $|\mathcal{O}_{\mathbb{P}(F)}(1)|$. Suppose $\phi$ is surjective, and is a divisorial contraction. Let $S\subsetneq W_0\subset \mathbb{P}(H^0(F))$ be subvarieties such that $\phi$ is an isomorphism outside $W_0$, and for all $w\in W_0\setminus S$, $\pi (\phi^{-1}(w)_{red})(\cong \phi^{-1}(w)_{red})$ is a linear subvariety of $\mathbb{P}^n$ of dimension $n+1-c_1$. Let $0\leq k \leq $ dim $W_0$ and $0\to \mathcal{O}_{\mathbb{P}^n}^k\xrightarrow{s} F\to \mathcal{E}\to 0 $ be a short exact sequence of vector bundles, and $L=\mathbb{P}(H^0(\mathcal{E}))\hookrightarrow \mathbb{P}(H^0(F))$ be the codimension $k$ linear subvariety corresponding to $s$. If $L\cap S=\varnothing$, then $(L\cap W_0)_{red}$ is a smooth subvariety of $L$ of codimension $n+2-c_1$, and $\phi|_{\mathbb{P}(\mathcal{E})}: \mathbb{P}(\mathcal{E})\to L$ is the blow up of $L$ along $(L\cap W_0)_{red}$.
 \end{lemma}
 \begin{proof}
 Since $k\leq$ dim $W_0$, we have $L\cap W_0\neq\varnothing$. So, $\psi=\phi|_{\mathbb{P}(\mathcal{E})}: \mathbb{P}(\mathcal{E})\to L$ is not an isomorphism.
 
 If $L\subset W_0$, then $$\hspace{-80 pt}\text{dim}(\mathbb{P}(F))-k=\text{dim}(\mathbb{P}(\mathcal{E}))=\text{dim} L+n+1-c_1$$$$=\text{dim}(\mathbb{P}(H^0(F)))-k +n+1-c_1=\text{dim}(\mathbb{P}(F))-k+n+1-c_1,$$ the last equality follows as $\text{dim}(\mathbb{P}(F))=\text{dim}(\mathbb{P}(H^0(F)))$, $\phi$ being birational. So, $c_1= n+1,$ contradicting $c_1\leq n$. So, $L\not \subset W_0$.

     So, $\psi$ is birational, and Ex($\psi)=$ Ex($\phi$) $\cap $ $\phi^{-1}(L)$ is a divisor in $\phi^{-1}(L)=\mathbb{P}(\mathcal{E})$, as $\phi^{-1}(L)\not \subset$ Ex($\phi$). So, if $R$ is the extremal ray of $\overline{NE}(\mathbb{P}(\mathcal{E}))$, which is not contracted by $\pi$, then $\psi$ is the contraction of $R$, and $\psi$ is divisorial. Let $w\in L\cap W_0$ and $C$ a line in $\psi^{-1}(w)_{red}(\cong \mathbb{P}^{n+1-c_1})$. Note that $\mathcal{O}_{\mathbb{P}(\mathcal{E})}(1).C=0$ as $C$ is contracted by the morphism given by $|\mathcal{O}_{\mathbb{P}(\mathcal{E})}(1)|$. Also, $\pi^*\mathcal{O}_{\mathbb{P}^n}(1).C=1$, as $C\xrightarrow{\pi} \mathbb{P}^n$ maps $C$ isomorphically to a line in $\mathbb{P}^n$. Since $-K_{\mathbb{P}(\mathcal{E})}=\pi^*\mathcal{O}_{\mathbb{P}^n}(n+1-c_1)\otimes \mathcal{O}_{\mathbb{P}(\mathcal{E})}(\text{rk }(\mathcal{E}))$, the length of $R$ is given by $l(R)=-K_{\mathbb{P}(\mathcal{E})}.C=n+1-c_1=\text{dim } \psi^{-1}(w).$ Now the Lemma follows by {\cite[Theorem 5.2]{andreatta2002special}}.
 \end{proof}
 \begin{lemma} \label{lemma:cute}
     Let $n$ be a positive integer, $F$ a globally generated vector bundle over $\mathbb{P}^n$ with $c_1(F)\leq n$, $\mathbb{P}(F)\xrightarrow{\phi}\mathbb{P}(H^0(F))$ the morphism given by $|\mathcal{O}_{\mathbb{P}(F)}(1)|$. Let $k\geq 0$, $0\to \mathcal{O}_{\mathbb{P}^n}^k\to F\to \mathcal{E}\to 0 $ be a short exact sequence of vector bundles, and suppose the other contraction of $\mathbb{P}(\mathcal{E})$ is a smooth blow up. Let $l$ be the largest integer such that $c_l(F)\neq 0$, and suppose $S\hookrightarrow \mathbb{P}(H^0(F))$ is a closed subset such that the following holds:
      for all $s\in S$, $\phi^{-1}(s)_{red}$ has a connected component which is not isomorphic to a projective space of dimension $\leq n-c_1(F)+1$,

     Then $\text{dim }S+1\leq k \leq \text{rk }F-l.$
 \end{lemma}
 \begin{proof}
     If $k>\text{rk }F-l,$ then there is a short exact sequence of vector bundles $0\to \mathcal{O}_{\mathbb{P}^n}^{\text{rk }F-l+1}\to F\to \mathcal{E}_1\to 0 $. By {\cite[Lemma 1.4]{anghel2013globally}}, $c_l(F)=0$, a contradiction. So $ k \leq \text{rk }F-l$. 
     
     To show $\text{dim }S+1\leq k$, it suffices to show $L\cap S=\varnothing$, where $L=\mathbb{P}(H^0(\mathcal{E}))\hookrightarrow \mathbb{P}(H^0(F))$ is the codimension $k$ linear subvariety. Note that $\psi=\phi|_{\mathbb{P}(\mathcal{E})}: \mathbb{P}(\mathcal{E})\to L$ is the morphism given by $|\mathcal{O}_{\mathbb{P}(\mathcal{E})}(1)|$, and $\psi^{-1}(y)\cong \phi^{-1}(y)$ for all $y\in L$.  Suppose the contrary, $L\cap S\neq\varnothing$. Let $ \mathbb{P}(\mathcal{E})\xrightarrow{\psi_1}Y\xrightarrow{\eta}L$ be the Stein factorization of $\psi$. Let $s\in L\cap S$. By assumption, $s$ is in image of $\psi$, so is in image of $\eta$. So, there is $y\in Y$ such that $\eta(y)=s$, and $\psi_1^{-1}(y)$ is not isomorphic to a projective space of dimension $\leq n+1-c_1$. 
     
     In particular dim $\psi_1^{-1}(y)>0$, so $\psi_1$ is not an isomorphism. So, $\psi_1$ is the other contraction of $\mathbb{P}(\mathcal{E})$, hence is a smooth blow up along 
 a subvariety $W$ of $Y$. Since $\psi_1$ is not an isomorphism, $|\mathcal{O}_{\mathbb{P}(\mathcal{E})}(1)|$ is not ample. By Lemma \ref{lemma:a=1}, $\text{codim }W=n+2-c_1$. So, $\psi_1^{-1}(y)=\mathbb{P}^{\text{codim }W-1}=\mathbb{P}^{n+1-c_1}$,a contradiction. So,  $L\cap S=\varnothing$, and we are done.
 \end{proof}
 \begin{lemma}\label{lemma:trivsum}
 Let $n$ be a positive integer and $\mathcal{E}_1$ a nonzero globally generated vector bundle over $\mathbb{P}^n$. If $\mathbb{P}(\mathcal{O}_{\mathbb{P}^n}\oplus \mathcal{E}_1)$ is a smooth blow up, then $E_1$ is direct sum of $\mathcal{O}_{\mathbb{P}^n}(1)$ with a trivial bundle.
 \end{lemma}
 \begin{proof}
     Let $\mathcal{E}=\mathcal{O}_{\mathbb{P}^n}\oplus \mathcal{E}_1$. Suppose $\mathbb{P}(\mathcal{E})$ is blow up of a smooth projective variety $Y$ along a smooth subvariety $W$. We are in the situation of the setup. Since $\mathcal{E}$ is nonample, we have codim $W=n+2-c_1$ by Lemma \ref{lemma:a=1}. By Lemma \ref{lemma:fiber}, the morphism given by $|\mathcal{O}_{\mathbb{P}(\mathcal{E})}(1)|$ has one fiber isomorphic to $\mathbb{P}^n$, so the same is true for contraction of $\mathbb{P}(\mathcal{E})$ given by Stein factorization of this morphism. It follows that the blow up of $Y$ along $W$ has a fiber over $Y$ of dimension $n$, hence codim $W=n+1$. So, $n+2-c_1=n+1$, so $c_1=1$. So, as $\mathcal{E}_1$ is globally generated, it must be of the form $\mathcal{O}_{\mathbb{P}^n}^{r-1}\oplus \mathcal{O}_{\mathbb{P}^n}(1)$ or $\mathcal{O}_{\mathbb{P}^n}^{r-1}\oplus T_{\mathbb{P}^n}(-1)$. If $n=1$ both are same, and if $n>1$ then the $2$nd case cannot occur as $s_n(\mathcal{E}_1)=s_n(\mathcal{E})>0$ by Lemma \ref{lemma:segre}.
 \end{proof}

 \begin{lemma}\label{lemma:ci}
 In the notations of the setup, suppose $Y=\mathbb{P}^n$. Then $W$ is nondegenerate and is not a complete intersection, unless $W$ is linear.
 \end{lemma}

\begin{proof}
    Suppose $W$ is nonlinear. The same argument as in {\cite[Proposition 2.5(a)]{ein1989some}} shows $W$ is nondegenerate. If $W$ is a complete intersection, then as in the proof of {\cite[Proposition 2.5(a)]{ein1989some}}, we have $n+1=h^0(\mathbb{P}^{n+r},\mathcal{I}_W(d))\leq n+r-m$, so $m-r+1\leq 0$. By Lemma \ref{lemma:a=1}, we get $c_1=m-r+2\leq 1$. Now as in the proof of Lemma \ref{lemma:trivsum} we conclude that $\mathcal{E}$ is direct sum of $\mathcal{O}_{\mathbb{P}^n}(1)$ with a trivial bundle, so $W$ is linear, a contradiction.
\end{proof}

\section{Examples}

Following are some examples of projective bundles over projective spaces which has a smooth blow up structure.
\begin{example}\label{0}
    For positive integers $m,n$, $X=\mathbb{P}_{\mathbb{P}^n}(\mathcal{O}^m\oplus \mathcal{O}(1))$ is blow up of $\mathbb{P}^{n+m}$ along a linear subvariety of dimension $m-1$.
\end{example} 
\begin{proof}
    This is well-known, see for example {\cite[Proposition 9.11]{eisenbud20163264}}.
\end{proof}
\begin{example}\label{1}
    Let $0\leq k \leq 2$, $W_0\hookrightarrow \mathbb{P}^5$ be the Segre embedding of $\mathbb{P}^1\times \mathbb{P}^2 $. Let $0\to \mathcal{O}_{\mathbb{P}^2}^k\to T_ {\mathbb{P}^2}(-1)^{\oplus2}\to E\to 0 $ be a short exact sequence of vector bundles. Then $X=\mathbb{P}_{\mathbb{P}^2}(E)$ is blow up of $H^{5-k}$ along $H^{5-k}\cap W_0$, where $H^{5-k}$ is a $(5-k)$-dimensional linear subspace in $\mathbb{P}^5$ such that $H^{5-k}\cap W_0$ is smooth and has codimension $2$ in $H^{5-k}$. For each $0\leq k\leq 2$, a general map $\mathcal{O}_{\mathbb{P}^2}^k\to T_{\mathbb{P}^2}(-1)^{\oplus2}$ gives such a short exact sequence.
\end{example} 
    \begin{proof}
        This is the content of \cite{ray2020projective}, also follows from Lemma \ref{lemma:example}.
    \end{proof}
     \begin{example}\label{2}
         Let $0\leq k \leq 3$, $W_0\hookrightarrow \mathbb{P}^9$ be the Plucker embedding of $Gr(2,5)$. Let $0\to \mathcal{O}_{\mathbb{P}^4}^k\to \wedge^2 T_{\mathbb{P}^4}(-1)\to E\to 0 $ be a short exact sequence of vector bundles. Then $X=\mathbb{P}_{\mathbb{P}^4}(E)$ is blow up of $H^{9-k}$ along $H^{9-k}\cap W_0$, where $H^{9-k}$ is a $(9-k)$-dimensional linear subspace in $\mathbb{P}^9$ such that $H^{9-k}\cap W_0$ is smooth and has codimension $3$ in $H^{9-k}$. For each $0\leq k\leq 3$, a general map $\mathcal{O}_{\mathbb{P}^4}^k\to \wedge^2 T_{\mathbb{P}^4}(-1)$ gives such a short exact sequence.
     \end{example} 
     \begin{proof}
         In Lemma \ref{lemma:example}, we take $F= \wedge^2 T_{\mathbb{P}^4}(-1)$. By {\cite[Theorem 1.2]{galkin2022projective}}, and the geometric description of $\phi$ in its proof, we see that $S=\varnothing,$ $W_0=$ Plücker embedding of Gr$(2,5)$ in $\mathbb{P}(\text{Alt}(5))\cong{P}^9$ satisfies the hypothesis of Lemma \ref{lemma:example}. So Lemma \ref{lemma:example} completes the proof of everything except the last statement. The last statement follows from {\cite[Lemma 1.4]{anghel2013globally}}, as $c_i(F)=0$ for $i\geq 4$.
     \end{proof}
     \begin{example}\label{3}
         Let $V=M_{3\times 4}(k)$, $W_0=\{[A]\in\mathbb{P}(V^*)| rkA\leq2\}, S=\{[A]\in\mathbb{P}(V^*)| rkA\leq1\}$. There is an identification $H^0(\mathbb{P}^3, T_{\mathbb{P}^3}(-1)^{\oplus3})=V, $ such that the following holds:
     
     Let $0\to \mathcal{O}_{\mathbb{P}^3}^6\xrightarrow{s} T_{\mathbb{P}^3}(-1)^{\oplus3}\to E\to 0 $ be a short exact sequence of vector bundles, such that $L\cap S=\varnothing,$ where $L=\mathbb{P}(H^0({E}))\hookrightarrow \mathbb{P}(V)$ is the codimension $6$ linear subspace corresponding to $s$.  Then $(L\cap W_0)_{red}$ is smooth and $X=\mathbb{P}_{\mathbb{P}^3}(E)$ is blow up of $L$ along $W:=(L\cap W_0)_{red}$. A general map $s$ gives such a short exact sequence with the above property.
     \end{example}} 
     \begin{proof}
         In Lemma \ref{lemma:example}, we take $F= T_{\mathbb{P}^3}(-1)^{\oplus3}$. Euler exact sequence gives short exact sequence $0\to \mathcal{O}_{\mathbb{P}^3}(-1)^3\to \mathcal{O}_{\mathbb{P}^3}^{4\times 3}\to T_{\mathbb{P}^3}(-1)^{\oplus3}$. This gives an identification $H^0(\mathbb{P}^3, T_{\mathbb{P}^3}(-1)^{\oplus3})=H^0(\mathbb{P}^3, \mathcal{O}_{\mathbb{P}^3}^{4\times 3}) =V$. Note that $S$ is projectively equivalent to the Segre embedding of $\mathbb{P}^2\times \mathbb{P}^3$, so dim $S=5.$ By {\cite[Theorem 1.1]{galkin2022projective}} and the geometric description of $\phi$ in its proof, $S$ and $W_0$ satisfies the hypothesis of Lemma \ref{lemma:example}. So Lemma \ref{lemma:example} completes the proof of everything except the last statement. The last statement follows from {\cite[Lemma 1.4]{anghel2013globally}}, as $c_i(F)=0$ for $i\geq 4$, and the fact that dim $S=5,$ so $L\cap S=\varnothing$ for a general linear subvariety $L$ of codimension $6$.
     \end{proof}
     \begin{example}\label{4}
         $i)$ For a positive integer $n$, $X=\mathbb{P}_{\mathbb{P}^n}(\mathcal{O}_{\mathbb{P}^n}(1) \oplus T_{\mathbb{P}^n}(-1))$ is the blow up of a smooth quadric in $\mathbb{P}^{2n+1}$ along a linear subvariety of dimension $n$.
     
     $ii)$ If 
\begin{equation*}
0\longrightarrow \mathcal{O}_{\mathbb{P}^n}\xlongrightarrow {s}{\mathcal{O}_{\mathbb{P}^n}(1) \oplus T_{\mathbb{P}^n}(-1)}\longrightarrow F\longrightarrow 0
\end{equation*}
 is an short exact sequence of vector bundles, where $s$ \medspace is a nowhere vanishing section, then $X=\mathbb{P}_{\mathbb{P}^n}(F)$ is the blow up of the smooth quadric in $\mathbb{P}^{2n}$ along a linear subvariety of dimension $n-1$.
     \end{example} 
 \begin{proof}
     This is {\cite[Corollary 3.3(1)]{bansal2024extremal}}.
 \end{proof}
 \begin{example}\label{5}
      Let $ n\geq 3$. Then:
     
(i) $X=\mathbb{P}_{\mathbb{P}^n}(\mathcal{O}_{\mathbb{P}^n}(1) \oplus \Omega_{\mathbb{P}^n}(2)) $ is the blow-up of $\emph{Gr}(2,n+2)$ along $\emph{Gr}(2,n+1)$.    

(ii) If $s$ \medspace is a nowhere vanishing section of $ \mathcal{O}_{\mathbb{P}^n}(1)\oplus \Omega_{\mathbb{P}^n}(2)$ and the vector bundle $E$ is defined by the exact sequence 
$$ 0\longrightarrow \mathcal{O}_{\mathbb{P}^n}(1) \longrightarrow \mathcal{O}_{\mathbb{P}^n}(1) \oplus \Omega_{\mathbb{P}^n}(2) \longrightarrow E \longrightarrow 0  ,$$ then $X= \mathbb{P}_{\mathbb{P}^n}(E)$ is the blow-up of the $ H \cap \emph{Gr}(2,n+2)$ along $ H \cap \emph{Gr}(2,n+1)$, where $ H \cap \emph{Gr}(2,n+2)$ is a smooth hyperplane section of $ \emph{Gr}(2,n+2)$ under the Plücker embedding, such that $ H\cap \emph{Gr}(2,n+1)$ is also smooth.
 \end{example}

 \begin{proof}
     This is {\cite[Corollary 3.3(2)]{bansal2024extremal}}.
 \end{proof}

 \begin{remark}
     Example \ref{0} is well-known, \ref{1} has appeared in \cite{ray2020projective}, \ref{4} and \ref{5} have appeared in \cite{bansal2024extremal}. \ref{2} with $k=0$ have appeared in {\cite[Theorem 2]{galkin2022projective}}. The rest of the examples are new, to the author's best knowledge.
 \end{remark}
 \begin{remark}
     $W$ in example \ref{3} is a codimension $2$ smooth subvariety of $\mathbb{P}^5$, which is not a complete intersection by Lemma \ref{lemma:ci}. Using arguments similar to {\cite[Example 3.1]{decker1995surfaces}} and {\cite[Table 7.3]{decker1995surfaces}}, one can show $W$ is a Bordiga 3-fold.
 \end{remark}
 \begin{remark}
     The bundle $E$ in example \ref{2} with $k=3$ is a Tango bundle on $\mathbb{P}^4$.
 \end{remark}
 \section{Proofs}

We first prove the following result.
 \begin{theorem} \label{theorem:cases}
     Let $n\geq 3, k\geq 0$ be integers, and $$0\to \mathcal{O}^k\to F\to \mathcal{E}\to 0$$ be a short exact sequence of vector bundles on $\mathbb{P}^n$, where $F$ is one of the following:
     \begin{enumerate}
         \item $p,q\geq 0$ are integers, $a_i,b_j$ are positive integers for $1\leq i\leq p, 1\leq j\leq q$, and $\sum_i a_i + \sum_j b_j\leq n$. $F=\oplus_i\mathcal{O}_{\mathbb{P}^n}(a_i)\oplus(\oplus_jP_{\mathcal{O}(b_j)}).$
         \item $F=\Omega_{\mathbb{P}^n}(2).$
         \item $F=\wedge^2(T_{\mathbb{P}^n}(-1))$
         \item $F=\mathcal{O}_{\mathbb{P}^n}(1)\oplus \Omega_{\mathbb{P}^n}(2).$
         \item $F=T_{\mathbb{P}^n}(-1)\oplus \Omega_{\mathbb{P}^n}(2).$
         \item $F=\mathcal{O}_{\mathbb{P}^n}(1)\oplus\wedge^2(T_{\mathbb{P}^n}(-1)).$
         \item $F=T_{\mathbb{P}^n}(-1)\oplus\wedge^2(T_{\mathbb{P}^n}(-1)).$
     \end{enumerate}

     Suppose the other contraction of $X=\mathbb{P}(\mathcal{E})$ is a smooth blow up and $\mathcal{E}(-1)$ is not nef. Then $X$ is one of the examples of \S 3.
 \end{theorem}
 \begin{proof}
 
 $(1)$ Suppose $F$ is an in $(1)$. Let $V_i=H^0(\mathbb{P}^n,{O}_{\mathbb{P}^n}(a_i)), W_j= H^0(\mathbb{P}^n,P_{\mathcal{O}(b_j)}), V=\oplus_iV_i, W=\oplus_j W_j$. $|\mathcal{O}_{\mathbb{P}(F)}|$ gives a morphism $\phi:\mathbb{P}(F)\to \mathbb{P}(V\oplus W) $. Let $S=\mathbb{P}(W)\hookrightarrow \mathbb{P}(V\oplus W).$ If $q=0$, then $\mathcal{E}(-1)$ is globally generated, hence nef, a contradiction. So, $q>0$, hence $c_n(F)\neq 0$. So, $l=n$ in the notations of Lemma \ref{lemma:cute}.

     \textbf{Case 1:} $\sum_i a_i+ \sum_j b_j>q+1$.

Using Lemma \ref{lemma:fiber}, note that for all $s\in S$, dim $\phi^{-1}(s)\geq n-q>n-(\sum_i a_i+ \sum_j b_j)+1=n-c_1(F)+1$. So, $S$ satisfies the hypothesis of Lemma \ref{lemma:cute}. By Lemma \ref{lemma:cute}, $\sum_j$dim $W_j=\text{dim }S+1\leq {rk }F-n= p+\sum_j$dim $W_j-q-n$, so $0\leq p-q-n$. But $n\geq \sum_i a_i + \sum_j b_j \geq p+q$. So, $q\leq 0$, a contradiction.

\textbf{Case 2:} $\sum_i a_i+ \sum_j b_j\leq q+1.$

We have $p+q\leq \sum_i a_i+ \sum_j b_j\leq q+1$. So, $p\leq 1$.

If $p=1,$ then $a_i=1=b_j$ for all $i,j$. If $p=0$, then either $b_j=1$ for all $j$, or $b_j=2$ for one $j$ and $b_j=1$ for all other $j$. So, $F=\mathcal{O}_{\mathbb{P}^n}(1)\oplus T_{\mathbb{P}^n}(-1)^q,T_{\mathbb{P}^n}(-1)^q $ or $T_{\mathbb{P}^n}(-1)^{q-1}\oplus P_{\mathcal{O}(2)}$.

\textbf{Subcase 1:} $F=\mathcal{O}_{\mathbb{P}^n}(1)\oplus T_{\mathbb{P}^n}(-1)^q$.

We can identify $W_j^*=k^{n+1}$ for all $j$, so $W^*=M_{q,n+1}$. Using Lemma \ref{lemma:fiber} one sees that $S=\{[0,A]\in \mathbb{P}(V\oplus W)| A\in W^*=M_{q,n+1},\text{ rk }A\leq q-1\}$ satisfies the hypothesis of Lemma \ref{lemma:cute}. We have dim $S+1=(q-1)(n+2)$, by the formula of dimension of determinantal variety, see for example the corresponding wikipedia page. Since $c_n(F)\neq 0$, we have $l=n$ in notation of Lemma \ref{lemma:cute}. By Lemma \ref{lemma:cute}, $(q-1)(n+2)\leq (q-1)n+1$, so $2(q-1)\leq 1$, So $q=1$, that is, $F=\mathcal{O}_{\mathbb{P}^n}(1)\oplus T_{\mathbb{P}^n}(-1)$ and $k\leq 1$ by Lemma \ref{lemma:cute}. Hence $X$ is as in example \ref{4}.

\textbf{Subcase 2:} $F=T_{\mathbb{P}^n}(-1)^q$.

We have $q\leq n$ by assumption. If $q<n$ then $s_n(\mathcal{E})=s_n(F)=0$, contradiction by Lemma \ref{lemma:segre}. So $q=n$. We can identify $W^*=M_{n,n+1}$ as in subcase 1. Using Lemma \ref{lemma:fiber} or the geometric description of $\phi$ in the proof of {\cite[Theorem 1]{galkin2022projective}}, one sees that $S=\{[A]\in \mathbb{P}(W)| A\in W^*=M_{n,n+1},\text{ rk }A\leq n-2\}$ satisfies the hypothesis of Lemma \ref{lemma:cute}. We have dim $S+1=(n-2)(n+3)$. Since $c_n(F)\neq 0$, we have $l=n$ in notation of Lemma \ref{lemma:cute}. By Lemma \ref{lemma:cute}, $(n-2)(n+3)\leq n^2-n$, so $n\leq 3$. So $n=3=q$, that is, $F= T_{\mathbb{P}^3}(-1)$. By Lemma \ref{lemma:cute} and its proof, we get $k= 6$ and $L\cap S=\varnothing$, where $L=\mathbb{P}(H^0(\mathcal{E}))\hookrightarrow \mathbb{P}(H^0(F))$ is the codimension $k$ linear subvariety. So, $X$ is as in example \ref{3}.

\textbf{Subcase 3:} $F=T_{\mathbb{P}^n}(-1)^{q-1}\oplus P_{\mathcal{O}(2)}$. 

By assumption, $q+1\leq n$. So, $s_n(\mathcal{E})=s_n(F)=0$, contradiction by Lemma \ref{lemma:segre}.

$(2)$ Suppose $F$ is as in $(2)$. So, $s_n(\mathcal{E})=s_n(F)=0$, contradiction by Lemma \ref{lemma:segre}.

$(3)$ Suppose $F$ is as in $(3)$. If $n$ is odd, $s_n(\mathcal{E})=s_n(F)=0$, contradiction by Lemma \ref{lemma:segre}. So $n$ is even.

As in \cite{galkin2022projective}, we can identify $H^0(F)=H^0(F)^*=Alt(n+1)$. By the geometric description of $\phi$ in the proof of {\cite[Theorem 2]{galkin2022projective}}, one sees that for all $s\in S:=\{[A]\in \mathbb{P}(Alt(n+1))| \text{ rk }A\leq n-4\}, \phi^{-1}(s)$ is a projective space of dimension $\geq 4$. So $S$ satisfies the hypothesis of Lemma \ref{lemma:cute}. We have dim $S+1=\frac{n^2+n}{2}-10$. Since $c_{n-1}(F)\neq 0, c_n(F)=0$, we have $l=n-1$ in notation of Lemma \ref{lemma:cute}. By Lemma \ref{lemma:cute}, $\frac{n^2+n}{2}-10\leq \frac{n^2-n}{2}-n+1$, so $n\leq 4$. As $n$ is even and $\geq 3$, we get $n=4$, and $k\leq 3$ by Lemma \ref{lemma:cute}. So, $X$ is as in example \ref{2}.

 $(4)$ Suppose $F$ is as in $(4)$. By Lemma \ref{lemma:cute}, $k\leq 1$. So $X$ is as in example \ref{5}.

 $(5)$ Suppose $F$ is as in $(5)$. By Lemma \ref{lemma:fiber}, for all $s\in S:=\mathbb{P}(H^0(T_{\mathbb{P}^n}(-1))\hookrightarrow \mathbb{P}(H^0(F))$, we have $\phi^{-1}(s)\cong \mathbb{P}^{n-1}$. So $S$ satisfies the hypothesis of Lemma \ref{lemma:cute}, as $n\geq 3$. We have dim $S+1=n+1$. Since $c_{n}(F)\neq 0$, we have $l=n$ in notation of Lemma \ref{lemma:cute}. By Lemma \ref{lemma:cute}, $n+1\leq n$, which is absurd.

 $(6)$ Suppose $F$ is as in $(6)$. Let $F_1=\wedge^2(T_{\mathbb{P}^n}(-1))$. By the same argument as in  {\cite[Theorem 2]{galkin2022projective}}, we can identify $H^0(F_1)=H^0(F_1)^*=Alt(n+1)$. Since $c_n(F)\neq0$, we have $l=n$ in notation of Lemma \ref{lemma:cute}.

 \textbf{Case 1:} $n$ is odd.

   By the same argument as in the geometric description of $\phi$ in the proof of {\cite[Theorem 2]{galkin2022projective}}, one sees that for all $s\in S:=\{[0,A]\in \mathbb{P}(H^0(\mathcal{O}(1)\oplus F_1))|A\in Alt(n+1), \text{ rk }A\leq n-3\}, \phi^{-1}(s)$ is a projective space of dimension $\geq 3$. So $S$ satisfies the hypothesis of Lemma \ref{lemma:cute}. We have dim $S+1=\frac{n^2+n}{2}-6$.  By Lemma \ref{lemma:cute}, $\frac{n^2+n}{2}-6\leq \frac{n^2-n}{2}-n+1$, so $n\leq 3$, hence $n=3$. But $\wedge^2(T_{\mathbb{P}^3}(-1))=\Omega_{\mathbb{P}^3}(2)$, so $F$ is an in $(5).$ We have already shown $X$ is as in the examples.

   \textbf{Case 2:} $n$ is even.

   By the same argument as in the geometric description of $\phi$ in the proof of {\cite[Theorem 2]{galkin2022projective}}, one sees that for all $s\in S:=\{[0,A]\in \mathbb{P}(H^0(\mathcal{O}(1)\oplus F_1))|A\in Alt(n+1), \text{ rk }A\leq n-2\}, \phi^{-1}(s)$ is a projective space of dimension $\geq 2$. So $S$ satisfies the hypothesis of Lemma \ref{lemma:cute}. We have dim $S+1=\frac{n^2+n}{2}-3$.  By Lemma \ref{lemma:cute}, $\frac{n^2+n}{2}-3\leq \frac{n^2-n}{2}-n+1$, so $n\leq 2$, a contradiction.

   $(7)$ Suppose $F$ is as in $(7)$. Let $F_1=\wedge^2(T_{\mathbb{P}^n}(-1))$. As in $(6)$ we can identify $H^0(F_1)=H^0(F_1)^*=Alt(n+1)$. Since $c_n(F)\neq0$, we have $l=n$ in notation of Lemma \ref{lemma:cute}.

   \textbf{Case 1:} $n$ is even.

   As we have seen in $(6)$, for all $s\in S:=\{[0,A]\in \mathbb{P}(H^0(T(-1)\oplus F_1))|A\in Alt(n+1), \text{ rk }A\leq n-2\}, \phi^{-1}(s)$ is a projective space of dimension $\geq 2$. So $S$ satisfies the hypothesis of Lemma \ref{lemma:cute}. We have dim $S+1=\frac{n^2+n}{2}-3$.  By Lemma \ref{lemma:cute}, $\frac{n^2+n}{2}-3\leq \frac{n^2-n}{2}$, so $n\leq 3$. A contradiction as $n$ is even and $n\geq 3$.

   \textbf{Case 2:} $n$ is odd.

By Euler exact sequence, identify $H^0(T(-1))=k^{n+1}.$
   Let $S=\{[\underline{\lambda}, A]\in \mathbb{P}(H^0(T(-1)\oplus H^0(F_1)))|\text{ rk }[A:\underline{\lambda}]\leq n-2\}=\{[\underline{\lambda}, A]\in \mathbb{P}(H^0(T(-1)\oplus H^0(F_1)))|\text{ rk }A\leq n-3\}$. As in Case $1$ of $(6)$, we get $n=3$. So, $F$ is as in $(5)$, which we already tackled.
\end{proof}

\textit{Proof of Theorem B:} We are as in the setup. Since $m\leq n+r-2$, we have $c_1\leq n$ by Lemma \ref{lemma:a=1}(2). Assume throughout that $\mathcal{E}$ is not as in example \ref{0}.

\underline{\textbf{(1):}} Let  $$0\to \mathcal{O}^k\to F\to \mathcal{E}\to 0$$ be a short exact sequence of vector bundles on $\mathbb{P}^n$ with $H^0(F^*)=H^1(F^*)=0$. By {\cite[Lemma 1.2]{anghel2013globally}} and Lemma \ref{lemma:trivsum}, there is always such a short exact sequence. Since $H^1(\mathbb{P}^n,\mathcal{O})=0$, this short exact sequence induces a short exact sequence on $H^0$. Since $\mathcal{O}^k$ and $\mathcal{E}$ are globally generated, an application of Snake Lemma shows $F$ is also globally generated. If $c_1(\mathcal{E})<n$ or if $c_1(\mathcal{E})=n$ and $F$ is decomposable, then since $c_1(F)=c_1(\mathcal{E})$ and $M_{n,c_1(\mathcal{E})}$ is true, we see that $F$ must be one of $(1)$-$(7)$ of Theorem \ref{theorem:cases}. So by Theorem \ref{theorem:cases}, $X$ is one of the examples.

Part \underline{\textbf{(2):}} will be proven in the end.

\textit{Proof of Theorem C:}  If $n\leq 2$ or $r=1$ then $X$ is one of the examples by \cite{szurek1990fano}, \cite{szurek1990fano1}, \cite{ancona1994fano}, \cite{szurek1990fano2}. So we assume $n\geq 3, r\geq 2.$ So $c_1(\mathcal{E}(-1))\leq 2$. If $\mathcal{E}(-1)$ is nef, then by{\cite[Lemma 3]{peternell2006numerically}}, $\mathcal{E}(-1)$ is globally generated. By {\cite[Theorem E]{bansal2024extremal}} (or by replacing $\mathcal{E}$ by $\mathcal{E}(-1)$ and proceeding with the proof below), $X$ is one of the examples. So assume $\mathcal{E}(-1)$ is not nef. Hence we are in the situation of the setup.  Assume throughout that $\mathcal{E}$ is not as in example \ref{0}.

\textbf{Case 1:} $\mathcal{E}$ is nonample.

Suppose $X$ is not one of the examples. Since $c_1\leq 5$, $M_{n,c_1}$ is true. By Theorem \ref{theorem:B}(1), we must have $c_1=n$, and if $0\to \mathcal{O}_{\mathbb{P}^n}^k\to F\to \mathcal{E}\to 0 $ a short exact sequence of vector bundles on $\mathbb{P}^n$ with $H^0(F^*)=H^1(F^*)=0$, then $F$ is indecomposable. Looking at the classification of globally generated vector bundles on $\mathbb{P}^n$ with first Chern class $\leq 5$ in \cite{anghel2013globally},\cite{anghel2020globally}, we see that one of the following cases occur:

\textbf{Subcase 1:} $F$ is as in $(xvi)$ of {\cite[Theorem 0.2]{anghel2013globally}}.

We have $s_4(\mathcal{E})=s_4(F)=0$ by a computation using exact sequences. Contradiction by Lemma \ref{lemma:segre}.

\textbf{Subcase 2:} $F$ is as in $(v)$ of {\cite[Theorem 0.1]{anghel2020globally}}.

We have $(-1)^is_i(\mathcal{E})=(-1)^is_i(F)=21$ for $3\leq i\leq 5$  by a computation using exact sequences. This is impossible by Lemma \ref{lemma:segre}.

\textbf{Subcase 3:} $F$ is as in $(vi)$ of {\cite[Theorem 0.1]{anghel2020globally}}.

We have $s_5(\mathcal{E})=s_5(F)=0$ by a computation using exact sequences. Contradiction by Lemma \ref{lemma:segre}.

\textbf{Case 2:} $\mathcal{E}$ is ample.

\textbf{Subcase 1:} $c_1\leq n-1.$

Since $c_1\leq 5$, $M_{n,c_1}$ is true. So, there is a short exact sequence of vector bundles $0\to \mathcal{O}_{\mathbb{P}^n}^k\to F\to \mathcal{E}\to 0 $ on $\mathbb{P}^n$ with $H^0(F^*)=H^1(F^*)=0$, and $F$ is one of $(1)$-$(7)$ as in Theorem \ref{theorem:cases}. By Theorem \ref{theorem:cases}, $X$ is as in the examples.

\textbf{Subcase 2:} $c_1= n=3.$

By {\cite[Theorem 0.1]{anghel2013globally}}, there is a short exact sequence of vector bundles $0\to \mathcal{O}_{\mathbb{P}^n}^k\to F\to \mathcal{E}\to 0 $ on $\mathbb{P}^n$ with $H^0(F^*)=H^1(F^*)=0$, and $F$ is one of $(1)$-$(7)$ as in Theorem \ref{theorem:cases}. By Theorem \ref{theorem:cases}, $X$ is as in the examples. But in each such case we have $\mathcal{E}$ nonample, a contradiction.

\textbf{Subcase 3:} $c_1= n=4.$

By {\cite[Theorem 0.2]{anghel2013globally}}, there is a short exact sequence of vector bundles $0\to \mathcal{O}_{\mathbb{P}^n}^k\to F\to \mathcal{E}\to 0 $ on $\mathbb{P}^n$ with $H^0(F^*)=H^1(F^*)=0$, and $F$ is one of $(1)$-$(7)$ as in Theorem \ref{theorem:cases}, or $F$ is as in $(xvi)$ of {\cite[Theorem 0.2]{anghel2013globally}}. In the first case, by Theorem \ref{theorem:cases}, $X$ is as in the examples. But in each such case we have $\mathcal{E}$ nonample, a contradiction. In the second case, we have We have $s_4(\mathcal{E})=s_4(F)=0$, so the top self-intersection number of $\mathcal{O}_{\mathbb{P}(\mathcal{E})}(1)$ is $0$, a contradiction as $\mathcal{O}_{\mathbb{P}(\mathcal{E})}(1)$ is ample.

\textbf{Subcase 4:} $c_1= n=5.$

Let $\mathcal{E}_1=\mathcal{E}\oplus\mathcal{O}_{\mathbb{P}^n}(1)$. So $\mathcal{E}_1$ is ample, $c_1(\mathcal{E}_1)=c_1(\mathbb{P}^n),$ rk $\mathcal{E}_1\geq 4 =n-1$. By {\cite[Theorem 0.1,0.2]{peternell1992fano}}, we see that $\mathcal{E}$ must be a split bundle. But then $\mathcal{E}(-1)$ is nef, a contradiction.

\textbf{Subcase 5:} $c_1= n+1.$

So $n\leq 4.$ We have rk $\mathcal{E}\geq 3\geq n-1$. If $n=4$, then rk $\mathcal{E}= 3$, and by {\cite[Theorem 7.4]{peternell1992fano}}, we see that we see that $X=\mathbb{P}(\mathcal{E})$ cannot be a smooth blow up, a contradiction. If $n=3$, by {\cite[Theorem 0.1]{peternell1992fano}}, we see that $X$ cannot be a smooth blow up, a contradiction.

\textbf{Subcase 6:} $c_1\geq n+2.$

As $n\geq 3, c_1\leq 5$, we must have $n=3, c_1=5$. Since $\mathcal{E}$ is ample, restricting $\mathcal{E}$ to a line in $\mathbb{P}^3$ we see that rk $\mathcal{E}\leq c_1(\mathcal{E})=5$. If rk $\mathcal{E}=5$, then restriction of $\mathcal{E}$ at every line must be $\mathcal{O}(1)^5$, hence restriction of $\mathcal{E}(-1)$ at every line is trivial. By {\cite[Theorem 3.2.1]{okonek1980vector}}, $\mathcal{E}(-1)$ is trivial, a contradiction as $\mathcal{E}(-1)$ is not nef. So, $3\leq $ rk $\mathcal{E}\leq 4$, that is, $r=2$ or $3$.

Suppose $r=3$, so by Lemma \ref{lemma:adivide}(ii), $a|{5\choose 3}=10$, so $a=2,5$ or $10$. So $r+1\nmid a+1$. So by Lemma \ref{lemma:a=1}(ii) we must have $n+r-m\geq 3$. By Lemma \ref{lemma:adivide}(iii), we have $a^2|{4\choose 3}=4$, so $a=2$. Since $0<b<a$, we have $b=1$. Now by Lemma \ref{lemma:adivide}(iv), we have $$4 \mid {5\choose 3}d+{5\choose 2}(1+d)-{5\choose 2}.10d,$$ hence $4|10d+10(1+d)=20d+10$, which is false. 

So, $r=2$. So, $m\leq n+r-2=3$. If $m\leq 2$, then by Lemma \ref{lemma:adivide}(iii), $a^2|{3\choose 3}=1$, contradiction to $a>1$. So, $m=3$. By Lemma \ref{lemma:a=1}(ii) and Lemma \ref{lemma:adivide}(ii), we have $3|a+1$ and $a|{4\choose 3}=4$. So, $a=2$. As $0<b<a$, we have $b=1$. By Lemma \ref{lemma:a=1}(v), we have $4=a^r|\alpha.$ So $Y$ is neither a projective space, nor a quadric in a projective space. By {\cite[Corollaries to Theorem 1.1 and 2.1]{kobayashi1973characterizations}}, we have $\tau\leq n+r-1=4.$ So by Lemma \ref{lemma:a=1}(ii), $\frac{d+3}{2}\leq 4$, so $d\leq 5$. Also, by Lemma \ref{lemma:adivide}(iv), $8|-50d+6$, so $4|d-3$. $d=3$ is the only possibility. So $d=3, \tau=3$, by Lemma \ref{lemma:a=1}(ii).

We have $0=EH_2^4=(-2H_1+3U)(-H_1+2U)^4=-48s_3-128s_2+608$, as $s_1=-c_1=-5$. So,
\begin{equation}
    3s_3+8s_2=38.
\end{equation} Also, $\alpha=H_2^5=(-H_1+2U)^5=-32s_3-80s_2+360$. So,
\begin{equation}
    2s_3+5s_2=\frac{45}{2}-\frac{\alpha}{8}.
\end{equation}

Solving equations $(3)$ and $(4)$ we get $s_2=9+\frac{3\alpha-8}{16}, s_3=-\frac{\alpha}{2}-10.$ Using equation $(1)$ we get $c_2=16-\frac{3\alpha-8}{16}, c_3=50-\frac{11\alpha}{8}.$

As $\mathcal{E}$ is ample. we have $c_3>0$. So, $\alpha\leq 36$. We also have $16|3\alpha-8$ as $c_2$ is an integer. Therefore $\alpha\equiv 8(\text{mod }16)$. So, $\alpha=8$ or $24$.

If $\alpha=24$, then $c_2=12,c_3=17$. But this is impossible as by {\cite[Equation 1.7]{anghel2018globally}}, $c_3-c_1c_2$ must be even.

If $\alpha=8$, then $c_2=15, c_3=39$. $\mathcal{G}:=P_{\mathcal{E}}$ is a globally generated vector bundle on $\mathbb{P}^3$ with $c_1(\mathcal{G})=5, c_2(\mathcal{G})=s_2(\mathcal{E})=10, c_3(\mathcal{G})=-s_3(\mathcal{E})=14$. But this is also impossible by {\cite[Theorem 0.1]{anghel2018globally}}.

So we get a contradiction.

 \textit{Proof of Theorem A:} 

 If $r=1$ or $c_1\leq 5$, then by \cite{ancona1994fano}, \cite{szurek1990fano} and Theorem \ref{theorem:C}, we see that $X$ must as in example 3. So we assume $r\geq 2, c_1\geq 6.$ By the already proven part of Theorem \ref{theorem:B}, $n\geq c_1\geq 6$.
 
 \underline{\textbf{(1):}} We have $d=2$, so $n$ is even, $c_1=\frac{n}{2}+1.$. As $c_1\geq 6$ we get $n\geq 10.$

 \textbf{Claim:} $n/2\leq r\leq n/2+3$.

 \begin{proof}
     By Lemma \ref{lemma:P} (iv), $n/2\leq r$. By Lemma \ref{lemma:a=1}(iii), $W$ is a quadratic variety in $\mathbb{P}^{n+r}$ in the sense of \cite{ionescu2013manifolds}. By \cite{ionescu2013manifolds} and Lemma \ref{lemma:ci}, we have $m\leq \frac{2}{3}(n+r)$, that is, $r\leq n/2+3.$
 \end{proof}

Let $r=\frac{n}{2}+g, 0\leq g\leq 3$. So, $m=\frac{n}{2}+g-1+\frac{n}{2}=n+g-1.$ 
By Lemma \ref{lemma:P} (ii), $P(t)=(1+t^{n/2})(1+t+\dots +t^{n/2+g-1})$. So, $a_i=1$ if $0\leq i\leq n/2-1$ or $n/2+g\leq i\leq m.$ Since $H^{2i}(W, \mathbb{Z})$ is free abelian by Lemma \ref{lemma:P}(iii), there are positive integers $l_i$ for integers $i\in [0,m]\setminus [n/2, n/2+g-1],$ such that $H^{2i}(W, \mathbb{Z})$ is generated by $H_2^i/l^i.$

Let $S=\{(k,i)\in \mathbb{Z}^2| 1\leq k\leq n+r, 0\leq i\leq min\{n/2-1, k-1\} \text{ and } k-i-1\in [0,m]\setminus [n/2,n/2+g-1]\}.$ In notations of {\cite[Theorem 7.31]{MR2451566}}, for $(k,i)\in S$, we have $H^{2k}(X, \mathbb{Z})\ni j_*\circ h^i\circ (\phi|_E)^*(H_2^{k-i-1}/l_{k-i-1})=j_*j^*((-E)^iH_2^{k-i-1}/l_{k-i-1})=(-1)^iH_2^{k-i-1}E^{i+1}/l_{k-i-1}=(-1)^iU^{k-i-1}(2U-H_1)^{i+1}/l_{k-i-1}.$ So,
\begin{equation*}
    l_{k-i-1}|U^{k-i-1}(2U-H_1)^{i+1} \text{ in }H^{2k}(X, \mathbb{Z}) \text{ for } (k,i)\in S.
\end{equation*}
 As $n\geq 10, 0\leq g\leq 3$, we have  $(n/2+g, g)\in S$. Since $n/2+g=r$ and $\{U^{r-i}H_1^i|0\leq i\leq r\}$ is a basis of the free abelian group $H^{2r}(X, \mathbb{Z})$, we get $l_{n/2-1}=1.$ By Poincare duality, $l_{n/2-1}l_{n/2+g-1}=l_m=e$, where $e$ is the degree of $W$ in $\mathbb{P}^{n+r}$. So, $l_{n/2+g}=e$.

 We also have $(n/2+g+1, 0)\in S$. So $$e| U^{n/2+g}(2U-H_1)=2U^{n/2+g+1}-U^{n/2+g}H_1$$ $$\hspace{50 pt}=2((n/2+1) U^rH_1+\sum_{i=2}^{r+1} (-1)^{i-1}c_iU^{r+1-i}H_1^i)-U^rH_1$$ $$\hspace{-10 pt}=(n+1)U^r H_1+\sum_{i=2}^{r+1}(-1)^{i-1}2c_iU^{r+1-i}H_1^i.$$ As $\{U^{r-i}H_1^{i+1} |0\leq i\leq r\}$ is a basis of the free abelian group $H^{2(r+1)}(X, \mathbb{Z})$, we have $e|n+1.$ As $W$ is nondegenerate by Lemma \ref{lemma:ci}, we have $e\geq \text{codim}(W)+1=n/2+1>(n+1)/2$. So we must have $e=n+1.$

As $n\geq 10$, we have $2\leq n/2+1\leq m-2.$ So $2\leq \text{ codim(W)}\leq m-2.$ Also, $e=n+1\geq n/2+4=\text{ codim }(W)+3$. By  {\cite[Proposition 2.4(2)]{park2024some}}, we have $n+1=e\geq 2\text{ codim }(W)+2=2(n/2+1)+2=n+4,$ a contradiction.
 
\underline{\textbf{(2):}} Suppose $d\geq 3$. By Lemma \ref{lemma:P} $(iv)$, we have $r\geq n/d$. Suppose $r\geq \frac{n}{d}+1$. So by Lemma \ref{lemma:a=1}(vi), we have $n+1\leq m+1.$ By Lemma \ref{lemma:a=1}(iv), $W$ is defined by $n+1$ equations in $\mathbb{P}^{n+r}$, and by Lemma \ref{lemma:ci}, $W$ is not a complete intersection. So, $(b)$ in {\cite[Main Theorem]{netsvetaev2006projective}} must be false. So, $m< \frac{3(n+r)}{4}-\frac{1}{2}.$ Hence, $\frac{d-1}{d}n+r-1<\frac{3}{4}(n+r)-\frac{1}{2}$, that is, $$ \Big(\frac{d-1}{d}-\frac{3}{4}\Big)n +\frac{r}{4}<\frac{1}{2}.$$ If $d\geq 4$, this implies $r<2$, a contradiction. So $d=3$ and $r/4<n/12 + 1/2.$ So, $r<n/3 + 2$. If $r=n/d+1$, then since $n\geq 6$, Lemma \ref{lemma:P}(ii) shows $a_{n/3}=2, a_{n/3+1}=1$. But by Hard Lefschetz theorem, the map $H^{2n/3}(W,\mathbb{C})\to H^{2n/3+2}(W,\mathbb{C})$ given by multiplication by a Kahler class in $H^2(W,\mathbb{C})$ is injective, so $a_{n/3}\leq a_{n/3+1}$, a contradiction. 

So, we have $r=n/d$. Now Lemma \ref{lemma:P}(ii) shows that $W$ has same Betti numbers as $\mathbb{P}^{n-1}.$ 

 \underline{\textbf{(3):}} Since Hartshorne's conjecture is true and by Lemma \ref{lemma:ci} $W$ is not a complete intersection, we have $m\leq\frac{2}{3}(n+r)$. Suppose the second case in $(2)$ holds, so $r=n/d, m=n-1$. So, $n-1\leq \frac{2(d+1)}{3d}n$, hence $2\leq r=n/d\leq \frac{3}{d-2}.$ This implies $d=3$ and $r=2$ or $3.$ If $r=2$, then $n=6$. So by Lemma \ref{lemma:a=1}, $ c_1=\frac{d-1}{d}n+1= 5$, a contradiction to our assumption that $c_1\geq 6$. So $r=3,n=9.$

So, $c_1=7=-s_1$. By Lemma \ref{lemma:segre} $s_9=-1, s_8=3, s_7=-9, s_6=27.$ Define integral polynomials $Q(t), R(t)$ by $Q(t)=1+7t+c_2t^2+c_3t^3+c_4t^4,$ $R(t)=1-7t+s_2t^2+s_3t^3+s_4t^4+s_5t^5.$ So, 
\begin{equation}
    Q(t)(R(t)+27t^6-9t^7+3t^8-t^9)\equiv 1(\text{mod } t^{10}).
\end{equation}
Hence,$$ Q(t)(R(t)-t^9)\equiv 1 (\text{mod } (3, t^{10})).$$ So, $Q(t)R(t)\equiv t^9+1\equiv (t+1)^9(\text{mod } (3, t^{10})).$ 

As $Q(t)R(t)-(t+1)^9$ has degree at most $9,$ we get $Q(t)R(t)\equiv (t+1)^2(\text{mod }3).$ As deg$Q\leq 4$ and deg$R\leq 5$, we must have $Q(t)\equiv (t+1)^4, R(t)\equiv (t+1)^5 (\text{mod }3)$, since $\mathbb{Z}/3\mathbb{Z}[t]$ is a unique factorization domain. So, there are $A(t), B(t)\in \mathbb{Z}[t]$ such that $Q(t)\equiv (1+t)^4+3A(t), R(t)\equiv (1+t)^5+3B(t)(\text{mod }9).$ 

By Equation $(6),$ $Q(t)(R(t)+3t^8-t^9)\equiv 1(\text{mod }(9, t^{10}))$. Using $c_1=7\equiv -2(\text{mod }9)$, we get $$Q(t)R(t)+3t^8-7t^9\equiv 1(\text{mod }(9,t^{10})).$$ So, $$Q(t)R(t)\equiv -2t^9-3t^8+1(\text{mod }(9,t^{10})).$$ As $Q(t)R(t)+2t^9+3t^8-1$ has degree at most $9$, we get $$Q(t)R(t)\equiv -2t^9-3t^8+1(\text{mod }9).$$
So, $$((1+t)^4+3A(t)))((1+t)^5+3B(t))\equiv -2t^9-3t^8+1(\text{mod }9).$$

Hence, $$-2t^9-3t^8+1\equiv0(\text{mod }(9,(1+t)^2)).$$ So, $(1+t)^2|-2t^9-3t^8+1=(1+t)S(t)$ in $\mathbb{Z}/9\mathbb{Z}[t]$, where $S(t)=-2t^8-t^7+t^6-t^5+t^4-t^3+t^2-t+1.$ As $1+t$ is a nonzerodivisor in $\mathbb{Z}/9\mathbb{Z}[t]$, we get $1+t|S(t)$ in $\mathbb{Z}/9\mathbb{Z}[t]$. So, $S(-1)\equiv 0(\text{mod }9).$ But $S(-1)=6$, a contradiction.

\begin{remark}
    It might be possible to prove that the 2nd case in Theorem \ref{theorem A}(2) is not possible without assuming Hartshorne's conjecture, using a calculation generalising the calculation in our proof of part $(3)$.
\end{remark}

\textit{Proof of Theorem B:} \underline{\textbf{(2):}} As in the beginning of the proof of Theorem \ref{theorem A}, it suufices to assume $n\geq c_1\geq 6$. Suppose rk $\mathcal{E}\geq n+1$, that is, $r\geq n$. So, $Y$ is covered by linear subvarieties $\phi(\pi^{-1}(x))$($x\in \mathbb{P}^n)$ of dimension $r\geq \frac{n+r}{2}.$  Since $Y$ has Picard rank $1$, by {\cite[Main Theorem]{sato1997projective}} and {\cite[Corollary 5.3]{novelli2011projective}}, one of the $3$ cases can occur:

\textbf{Case 1:} $Y$ is a projective space.

By Theorem \ref{theorem A} we are done.

\textbf{Case 2:} $Y=Gr(2, n+2).$

As in {\cite[\S 6]{sato1997projective}}, $\mathbb{P}^{n+1}$ gives a family of $n$-planes in $Y$ with universal family $\mathbb{P}_{\mathbb{P}^{n+1}}(\Omega_{\mathbb{P}^{n+1}}(2)).$ Since $\Omega_{\mathbb{P}^{n+1}}(2)$ and $\mathcal{E}$ both has first Chern class $n$, the family of $n$-planes in $Y$ given by $\phi(\pi^{-1}(x))$($x\in \mathbb{P}^n)$ is induced by a linear map $f:\mathbb{P}^{n}\to \mathbb{P}^{n+1}$. In other words, $\mathcal{E}=f^*(\Omega_{\mathbb{P}^{n+1}}(2))=\mathcal{O}_{\mathbb{P}^n}(1) \oplus \Omega_{\mathbb{P}^n}(2).$ So we are as in example \ref{5}.

\textbf{Case 2:} $Y=Q_{2n},$ the $2n$-dimensional smooth quadric.

Quotienting $\mathcal{E}$ by a general section, we get a nef but nonample vector bundle $\mathcal{E}_1$ of rank $n$ such that $\mathbb{P}(\mathcal{E}_1)$ is blow up of $Q_{2n-1}$ along a smooth codimension $2$ subvariety $W_1$. Since $Q_{2n-1}$ and $\mathbb{P}_{2n-1}$ have the same Betti numbers, the same argument as in Lemma \ref{lemma:P} (ii) shows $a_1=2.$ But $W_1$ is a smooth codimension $3$ subvariety of $\mathbb{P}_{2n}$, so by Barth-Larsen theorem we get $2>2n-6$, hence $n\leq 3$, a contradiction.

\begin{remark} \label{rmk}
A similar proof as in Theorem \ref{theorem:B}(1) shows the following:

Let $n,r$ be positive integers, $X$ a smooth projective variety which is a $\mathbb{P}^r$-bundle over $\mathbb{P}^n$, and has another projective bundle structure over some smooth projective variety $Y$. Let $\mathcal{E}$ be the unique nef vector bundle of rank $r+1$  over $\mathbb{P}^n$ such that $X\cong \mathbb{P}_{\mathbb{P}^n}(\mathcal{E})$ over $\mathbb{P}^n$ and $\mathcal{E}(-1)$ is not nef. Suppose $\mathcal{E}$ is not ample. Then $c_1(\mathcal{E})\leq n-1$ and the following holds:
     Suppose $\mathcal{E}$ is globally generated, equivalently, the ample generator of $Y$ is globally generated. If $M_{n,c_1(\mathcal{E})}$ is true, then $\mathcal{E}$ is either trivial, or $T_{\mathbb{P}^n}(-1)$, or $\Omega_{\mathbb{P}^n}(2)$, or a null-correlation bundle.
 \end{remark}

\section{Acknowledgement}
I am grateful to Professor János Kollár for giving valuable ideas, guidence and references. I also thank Varun Rao, Shivam Vats and Ashima Bansal for helping me with Latex.
\printbibliography
\end{document}